\documentclass[a4paper]{amsart}

\usepackage{amssymb}
\usepackage{mathrsfs}
\usepackage[matrix,arrow,curve,tips]{xy}\SelectTips{cm}{}

\newcommand{\NN}{\mathbb{N}}
\newcommand{\ZZ}{\mathbb{Z}}
\newcommand{\RR}{\mathbb{R}}

\newcommand{\Hol}{\textnormal{Hol}}
\newcommand{\fol}{\mathcal{F}}

\newcommand{\GPD}{\textsf{GPD}}
\newcommand{\Eff}{\textnormal{Eff}}

\newcommand{\Gg}{\mathscr{G}}
\newcommand{\Gh}{\mathscr{H}}
\newcommand{\Gk}{\mathscr{K}}

\newenvironment{Para}[1]{\vspace{0.25cm}\noindent\ignorespaces\textbf{#1}}{\hspace{\stretch{1}}\vspace{0.25cm}}
\newtheorem{Theorem}{Theorem}[section]
\newtheorem{Proposition}[Theorem]{Proposition}
\newtheorem{Lemma}[Theorem]{Lemma}

\newtheorem{Definition}[Theorem]{Definition}
\newtheorem{Example}[Theorem]{Example}
\newtheorem{Remark}[Theorem]{Remark}

\begin{document}

\begin{abstract}
In this paper, we  generalize the notion of
Serre fibration to the Morita category of topological groupoids and derive
the associated long exact sequence of homotopy groups.  We use this results
for calculation of homotopy groups of various groupoids,  such as
the foliation groupoid of a Riemannian foliation.
\end{abstract}
  
\title{Serre fibrations in the Morita category of topological groupoids}
\author{Bla\v{z} Jelenc}
\address{Institute of Mathematics, Physics and Mechanics,
         University of Ljubljana, Jadranska 19,
         1000 Ljubljana, Slovenia}
\email{blaz.jelenc@imfm.si}
\thanks{This work was supported in part by
        the Slovenian Research Agency (ARRS) project J1-2247.}

\maketitle

\section{Introduction}

Topological groupoids are objects that we can use to
represent singular geometric structures.
For example, the leaf space of a foliation may not contain
much information about the foliation, but it is well represented
by the associated holonomy groupoid \cite{IntroFolandLie}.
By reducing this groupoid to a complete
transversal of the foliation, we obtain a different groupoid which
represents the foliation as well as the holonomy groupoid and
is Morita equivalent to the holonomy groupoid.
Morita equivalence is a relation between topological groupoids,
generated by the functors between topological groupoids which are
called weak equivalences. Such functors are, in particular,
equivalences of categories and induce homeomorphisms between the spaces of orbits
of the topological groupoids \cite{Moerdijk1,IntroFolandLie,HilSkan}.
There is the Morita category of topological groupoids
in which two topological groupoids are isomorphic precisely if they are Morita equivalent.
The morphisms in this category can be represented by principal bundles.

It is essential to have a way of distinguishing topological
groupoids that are not Morita equivalent.
There are some interesting classes of topological groupoids that 
are closed under Morita equivalence, such as proper groupoids and
foliation groupoids \cite{IntroFolandLie}. 
A classical way to construct invariants of topological groupoids is based
on the classifying space associated to a topological groupoid,
which is determined uniquely up to homotopy equivalence.
There are basically two standard constructions of this space,
the Milnor's infinite join construction
and geometric realization of the nerve of topological groupoid
\cite{Bott,Dold,Haefliger,FibBund,ClassTopoi,Moerdijk1,DiffGeoFol,Tsuboi}.
One can show that a weak equivalence between topological groupoids induces a weak homotopy
equivalence between the associated classifying spaces.
We can therefore study the homotopy invariants of the classifying space as
the Morita invariants of the topological groupoid.
With this approach, some very interesting results can be obtained.
For  example, the Haefliger theorem tells us that classifying space
of the holonomy  groupoid of a foliation on a manifold $M$ is homotopy equivalent to $M$
if and  only if the holonomy cover of every leaf of the foliation is contractible \cite{HaefligerClass}.
In general, however, the classifying space of a topological groupoid is a complicated space
and can be difficult to understand and use. Furthermore, it can be deprived
of some additional geometric structure which is present on the level of groupoids,
such as the smooth structure of the holonomy groupoid.
An alternative approach is to define Morita invariants of topological groupoids
in a manner that does not directly rely on the classifying space.
For example, one can describe the fundamental group of an \'{e}tale groupoid $\Gg$ in
terms of $\Gg$-paths \cite{HilSkan} or the fundamental group of an orbifold
$Q$ in terms of deck transformations of the universal cover of $Q$
\cite{Thurston} (for higher homotopy groups of orbifolds, see also \cite{Chen}).
  
In this paper we study the homotopy groups of topological groupoids and their
description which is intrinsic in the Morita category  and does not involve
the classifying space. To achieve this, we first generalize
the Morita category of topological groupoids to the Morita category of pairs of 
topological groupoids. In this category, the
$n$-th homotopy group $\pi_n(\Gg,a_0)$
of a topological groupoid $\Gg$ with base point $a_0$
can be described as the set of homotopy classes of
Morita maps from the pair $(I^n,\partial I^n)$ to the pointed topological groupoid
$(\Gg,a_0)$, with natural multiplication induced by concatenation \cite{HilSkan}.
We show that this group is isomorphic to the $n$-th homotopy group of
the associated classifying space.
With our definition, is straightforward that $\pi_{n}$
is a functor, defined on the Morita category of topological groupoids
and therefore a Morita invariant.

One of our objectives is to develop some methods for efficient calculation
of homotopy groups of topological groupoids. We show that for every
Morita map $P$ between pointed topological groupoids $(\Gh,b_0)$ and $(\Gg,a_0)$
there is a natural long exact sequence of groups
$$
\ldots\to
\Sigma_n(P)\to
\pi_n(\Gh,b_0)\stackrel{\pi_n(P)}{\to}
\pi_n(\Gg,a_0)\to
\Sigma_{n-1}(P)\to
\ldots
$$ 
for some groups $\Sigma_n(P)$ depending on $P$ which we describe explicitly in terms
of principal bundles. Furthermore, we generalize the notion of Serre fibration
to the Morita category of topological groupoids
(we emphasize that this definition is intrinsic in the Morita bicategory of topological groupoids).
We show that for a Morita map $P$
that is a Serre fibration, one can identify the groups $\Sigma_n(P)$ with the homotopy
groups of some topological groupoid, which is naturally associated to the Morita map $P$ and
can be viewed as some sort of a fiber of $P$. 
  
In the rest of the paper we apply these results to many
concrete examples. We consider a class of so called Serre groupoids 
(which have the property that their source map is a Serre fibration in the classical sense)
and show that the calculation of homotopy groups of
such groupoids is especially simple.
Examples of Serre groupoids are action groupoids and holonomy
groupoids of some foliations. 
  
It will turn out that there are many functors between topological groupoids
that are Serre fibrations as Morita maps.
For instance, the functor from the holonomy groupoid
of the foliation of transverse principal bundle to the holonomy groupoid of the
foliation of the base space is a Serre fibration.
This turns out to be useful in the calculation of
homotopy groups of the holonomy groupoid of Riemannian foliation on a compact
manifold. Another important example of a Serre fibration is the effect functor $\Eff$
from a proper \'{e}tale groupoid $\Gg$ to the associated proper effective \'{e}tale groupoid
$\Eff(\Gg)$. We will prove that the functor $\Eff$ induces an isomorphism on the $n$-th
homotopy group for $n\geq 3$, whereas the relation between the first and the second
homotopy groups of $\Gg$ and $\Eff(\Gg)$ is expressed in
terms of an exact sequence.

\section{Topological groupoids and principal bundles}\label{section:sec1}
  
\textsl{
In this section we first recall some definitions and facts
of the theory of topological groupoids and principal bundles.
We introduce the Morita category of pairs of topological groupoids and
give the definition of homotopy groups of topological groupoids in terms of
Morita maps. }
\vspace{0.25cm}
  
Throughout this notes, we will say that a (continuous)
map $f:X\to Y$ between topological spaces has local
sections if for any point $y\in Y$ there exists a map
$\sigma:V\to X$, defined on an open neighbourhood $V$ of $y$ in $Y$, such
that $f\circ\sigma=id_V$. It is important to note that any map
with local sections is a quotient map. More generally, a map
of topological pairs $f:(X,A)\to (Y,B)$ has local sections
if for any point $y\in Y$ there exists a map of topological pairs
$\sigma:(V, V\cap B)\to (X,A)$,
defined on an open neighbourhood $V$ of $y$ in $Y$, such
that $f\circ\sigma=id_V$.

\Para{Topological groupoid.} A groupoid $\Gg$ is a small category
such that every arrow in $\Gg$ is invertible. A topological groupoid is a
groupoid $\Gg$ such that  its space of arrows $\Gg_1$ and its space of
objects $\Gg_0$ are both  topological spaces and all of the structure maps of
$\Gg$ are continuous. Note that, in particular, the source map
$s:\Gg_{1}\to\Gg_{0}$ and the target map
$t:\Gg_{1}\to\Gg_{0}$ both have local sections - in fact,
they both have a global section
$\textnormal{uni}:\Gg_{0}\to\Gg_{1}$, which maps an object
$x$ of $\Gg$ to the unit arrow $1_{x}$ at $x$.
A topological groupoid $\Gg$ is often denoted also by
$(\Gg_1\rightrightarrows \Gg_0)$.

First examples of topological groupoids are topological spaces:
any topological space $X$ can be viewed as a (unit) topological groupoid
$(X\rightrightarrows X)$. Any topological group $G$ is a topological groupoid
with only one object, thus $(G\rightrightarrows \ast)$.
The product of topological groupoids $\Gg$ and $\Gh$
is the topological groupoid $((\Gg_1\times\Gh_1)\rightrightarrows(\Gg_0\times\Gh_0))$.

\Para{Groupoid action.} Let $\Gg$ be a topological groupoid and $X$ a
topological space. A right $\Gg$-action on $X$ along a map $\epsilon:X\to \Gg_0$
is a map $\mu:X\times_{\Gg_0}\Gg_1\to X$ (we write $\mu(x,g)=xg$),
defined on the pullback
$X\times_{\Gg_0}\Gg_1=\left\{(x,g)|\epsilon(x)=t(g)\right\}$,
such that
\begin{enumerate}
  \item [(i)]   $\epsilon(xg)=s(g)$,
  \item [(ii)]  $x 1_{\epsilon(x)}=x$, and
  \item [(iii)] $(xg)g'=x(gg')$ for any $x\in X$, $g,g'\in \Gg_1$ with
                $\epsilon(x)=t(g)$ and $s(g)=t(g')$.
\end{enumerate}
The space of orbits of such an action is denoted by $X/\Gg$.
For such an action $\mu$ we have the associated translation groupoid
$X\rtimes\Gg=(X\times_{\Gg_0}\Gg_1 \rightrightarrows X)$, for which
the source map is $\mu$ and the target map is the first projection.
The multiplication $X\rtimes\Gg$ is determined by
$(x,g)(x',g')=(x,gg')$.

Analogously, one defines a left action of a topological groupoid
$\Gh$ on $X$ and the associated translation groupoid $\Gh\ltimes X$.

\Para{Principal bundle.} Let $\Gg$ and $\Gh$ be topological groupoids. A
$\Gg$-principal bundle over $\Gh$ is a space $P$, equipped with a left
$\Gh$-action along a map $\pi:P\to \Gh_0$ and
a right $\Gg$-action along
a map $\epsilon:P\to \Gg_0$
such that
\begin{enumerate}
\item [(i)]   the map $\pi:P\to \Gh_0$ has local sections,
\item [(ii)]  for every $p\in P$ and $g\in \Gg_1$ with $\epsilon(p)=t(g)$  we have $\pi(pg)=\pi(p)$,
\item [(iii)] for every $p\in P$ and $h\in \Gh_1$ with $s(h)=\pi(p)$  we have $\epsilon(hp)=\epsilon(p)$,
\item [(iv)]  the actions of $\Gh$ and $\Gg$ on $P$ commute with each other,
              i.e. $h(pg)=(hp)g$ for every $h\in\Gh_1$, $p\in P$ and $g\in\Gg_1$ with
              $s(h)=\pi(p)$ and $\epsilon(p)=t(g)$, and
\item [(v)]   the map $\mu:P\times_{\Gg_0}\Gg_1\to P\times_{\Gh_0} P$, $(x,g)\mapsto(x,xg)$,
              is a homeomorphism.
\end{enumerate}
Any principal $\Gg$-bundle over $\Gh$ induces a map
$\vartheta:P\times_{\Gh_0} P\to \Gg_1$, called the translation map,
uniquely determined with the 
property that $p'=p\,\vartheta(p,p')$.
Also note that the space $\Gh_0$
is homeomorphic to the orbit space $P/\Gg$,
because the map $\pi$ has local sections and is therefore quotient map.
  
We will say that a principal $\Gg$-bundle $P$ over $\Gh$ is numerable, if the
map $\pi:P\to\Gh_0$ has sections over a numerable covering of $\Gh_0$.
Recall from \cite{Dold,FibBund} that a covering of a topological space is
numerable if it is a open covering with subordinated locally finite partition
of unity.

Let $P$ and $P'$ be  principal $\Gg$-bundles over $\Gh$. A map $\varphi:P\to
P'$ is a morphism of principal $\Gg$-bundles over $\Gh$ if
$\epsilon=\epsilon'\circ\varphi$, $\pi=\pi'\circ\varphi$ and $\varphi$ is
$\Gh$-$\Gg$-equivariant, namely
$\varphi(hpg)=h\varphi(p)g$
for every $h\in\Gh_1$, $p\in P$ and $g\in\Gg_1$ with $s(h)=\pi(p)$
and $\epsilon(p)=t(g)$. Any morphism of principal $\Gg$-bundles
over $\Gh$ is a homeomorphism (thus an isomorphism) \cite[p.165]{PoissonGeo}.
  
For a principal $\Gg$-bundle $P$ over $\Gh$ and a principal $\Gh$-bundle $Q$
over $\Gk$, there is a well defined tensor product $Q\otimes P=Q\otimes_{\Gh} P$ which
is a principal $\Gg$-bundle over $\Gk$.
The space $Q\otimes P$ is the space of
orbits of the diagonal $\Gh$-action on the fibered product
$Q\times_{\Gh_0}P$ \cite{PoissonGeo,HilSkan}.
The Morita category
$$\GPD $$
of topological
groupoids is the category with topological groupoids as objects and
isomorphism classes of principal bundles as morphisms,
and with composition induced by the tensor product.
(Considering the principal bundles, rather than their isomorphism classes,
as $1$-morphisms, and morphisms of principal bundles as $2$-morphisms,
one obtains the Morita bicategory of topological groupoids.)
The category of topological spaces
can be regarded as a full subcategory of the Morita category $\GPD$
in the obvious way.
  
\Para{Pairs of topological groupoids.} A pair of topological groupoids is a
pair $(\Gg,\Gg')$ such that $\Gg$ is a topological groupoid
and $\Gg'$ a subgroupoid of $\Gg$.
Let $(\Gh,\Gh')$ and $(\Gg,\Gg')$ be pairs of topological groupoids.
A principal $(\Gg,\Gg')$-bundle
over $(\Gh,\Gh')$ is a pair of topological spaces $(P,P')$
such that $P$ is a principal $\Gg$-bundle over $\Gh$,
$P'$ is a principal $\Gg'$-bundle over $\Gh'$ with respect to the structure
given by the restriction of the structure of $P$ to $P'\subset P$
and the map of topological pairs $\pi:(P,P')\to (\Gh_0,\Gh'_0)$
has local sections.

We will say that the principal $(\Gg,\Gg')$-bundle
$(P,P')$ over $(\Gh,\Gh')$ is numerable if the map $\pi:(P,P')\to(\Gh_0,\Gh'_0)$
has sections (as maps of topological pairs) over a numerable covering. 

A morphism $\phi$ from a principal $(\Gg,\Gg')$-bundle $(P,P')$ over
$(\Gh,\Gh')$ to a principal $(\Gg,\Gg')$-bundle $(R,R')$ over
$(\Gh,\Gh')$ is a morphism $\phi:P\to R$
of principal $\Gg$-bundles over $\Gh$
such that $\phi(P')\subset R'$.
 
Let $(P,P')$ be a principal $(\Gg,\Gg')$-bundle over $(\Gh,\Gh')$
and $(Q,Q')$ a principal $(\Gh,\Gh')$-bundle over $(\Gk,\Gk')$.
Then the tensor product
$(Q,Q')\otimes (P,P')=(Q\otimes P,Q'\otimes P')$
is a principal $(\Gg,\Gg')$-bundle
over $(\Gk,\Gk')$ (note that we can naturally identify $Q'\otimes P'$
with a subset of $Q\otimes P$). 

The Morita category
$$\GPD^2 $$
of pairs of topological
groupoids is the category with pairs of topological groupoids as objects and
isomorphism classes of principal bundles as morphisms,
and with composition induced by the tensor product.
The category of pairs of topological spaces
can be regarded as a full subcategory of the Morita category $\GPD^2$.

\Para{Pointed and marked topological groupoids.} Let $\Gg$ be a topological groupoid
and $A$ a subset of $\Gg_0$.
We can view the  set $A$ as a  unit subgroupoid of $\Gg$, namely the
subgroupoid $A=\textnormal{uni}(A)\subset \Gg$.
In this view, the pair $(\Gg,A)$ is a pair of topological groupoids,
called a marked topological groupoid. If $A$ is a singleton set $\left\{a_0\right\}$,
then the marked topological groupoid $(\Gg,\left\{a_0\right\})$
is denoted simply by $(\Gg,a_0)$ and called a pointed topological groupoid.
If $(P,P')$ is a principal $(\Gg,A)$-bundle over $(\Gh,\Gh')$,
then we will rather see $P'$ as a section $\sigma$ of $\pi:P\to\Gh_0$ over $\Gh_0'$,
and write $(P,P')=(P,\sigma)$. 

The full subcategory of $\GPD^2$, consisting of all marked
topological groupoids, is the Morita category of marked topological groupoids
and denoted by
$$\GPD_{\scriptscriptstyle\square}.$$
The full subcategory of $\GPD_{\scriptscriptstyle\square}$, consisting of all pointed
topological groupoids, is the Morita category of pointed topological groupoids
and denoted by
$$\GPD_\circ.$$
The category of pairs of topological spaces
is a full subcategory of the Morita category $\GPD_{\scriptscriptstyle\square}$,
while the category of pointed topological spaces is a full subcategory of
the Morita category $\GPD_\circ$.

\Para{Morita equivalence.}
A functor $\phi:(\Gh,b_0)\to(\Gg,a_0)$ between pointed 
groupoids gives us an associated principal $(\Gg,a_0)$-bundle
$(\left\langle \phi\right\rangle,(b_0,1_{a_0}))$
over $(\Gh,b_0)$, given as the pullback
$$
\left\langle \phi\right\rangle=\Gh_0\times_{\Gg_0}\Gg_1=\left\{(y,g)|\phi(y)=t(g)\right\}.
$$
The space $\left\langle \phi\right\rangle$ has a right
$\Gg$-action along the map
$s\circ\textnormal{pr}_2:\left\langle \phi\right\rangle\to\Gg_0$
and a left $\Gh$-action along the map
$\textnormal{pr}_1:\left\langle \phi\right\rangle\to\Gh_0$
given in the obvious way.
The functor $\phi$ is a weak equivalence if the map
$s\circ pr_2:\left\langle \phi\right\rangle\to \Gg_0$
is a map with local sections and the diagram
$$
  \xymatrix{
  \Gh_1\ar[r]^-\phi\ar[d]_-{(s,t)}& \Gg_1\ar[d]^-{(s,t)}\\
  \Gh_0\times \Gh_0\ar[r]_-{\phi\times\phi}& \Gg_0\times \Gg_0
  }
$$
is a pullback of topological spaces. 
The functor $\phi$ is a weak
equivalence if and only if the isomorphism class of the bundle
$\left\langle\phi\right\rangle$ is invertible in $\GPD$ \cite{HilSkan}.
In this case, the principal $(\Gg,a_0)$-bundle
$(\left\langle \phi\right\rangle,(b_0,1_{a_0}))$
over $(\Gh,b_0)$ is an isomorphism in the Morita category of
pointed topological groupoids $\GPD_\circ$.
  
In general, an isomorphism class of a principal $(\Gg,a_0)$-bundle over
$(\Gh,b_0)$ is called Morita equivalence
if it is an isomorphism in $\GPD_\circ$.

\Para{Homotopy.} We say that two principal $\Gg$-bundles $P$
and $Q$ over $\Gh$ are homotopic,
if there exists a principal $\Gg$-bundle $E$  over
$\Gh\times I$ such that $P$ is isomorphic to the restriction
$E|_{\Gh\times \left\{0\right\}}$ and
$Q$ is isomorphic to the restriction
$E|_{\Gh\times\left\{1\right\}}$.
(The restriction $E|_{\Gh\times \left\{0\right\}}$ 
of $E$ to the subgroupoid $\Gh\times \left\{0\right\}$
of $\Gh\times I$, naturally isomorphic to $\Gh$,
is defined in the obvious way.)

More generally, two principal $(\Gg,\Gg')$-bundles $(P,P')$
and $(Q',Q')$ over $(\Gh,\Gh')$ are homotopic,
if there exists a principal $(\Gg,\Gg')$-bundle $(E,E')$  over
$(\Gh\times I,\Gh'\times I)$,
such that $(P,P')$ is isomorphic to the restriction
$(E|_{\Gh\times \left\{0\right\}},E'|_{\Gh'\times \left\{0\right\}})$ and
$(Q,Q')$ is isomorphic to the restriction
$(E|_{\Gh\times \left\{1\right\}},E'|_{\Gh'\times \left\{1\right\}})$.

\begin{Proposition}\label{prop:homotopy}
Homotopy is an equivalence relation on the set of principal
$(\Gg,\Gg')$-bundles over $(\Gh,\Gh')$.
\end{Proposition}
  
\begin{proof}
To prove that the relation is transitive,
let $(E,E')$ be a
principal $(\Gg,\Gg')$-bundle over $(\Gh\times I,\Gh'\times I)$, $(F,F')$
a principal $(\Gg,\Gg')$-bundle over $(\Gh\times I,\Gh'\times I)$
and $(P,P')$ a principal $(\Gg,\Gg')$-bundle over $(\Gh,\Gh')$
which is isomorphic to both
$(E|_{\Gh\times \left\{1\right\}},E'|_{\Gh'\times\left\{1\right\}})$ and
$(F|_{\Gh\times \left\{0\right\}},F'|_{\Gh'\times\left\{0\right\}})$.
We cannot directly glue the bundles $(E,E')$ and $(F,F')$
together along the isomorphism
between the restrictions 
$(E|_{\Gh\times \left\{1\right\}},E'|_{\Gh'\times\left\{1\right\}})$ and
$(F|_{\Gh\times \left\{0\right\}},F'|_{\Gh'\times\left\{0\right\}})$
because the resulting bundle may not have local sections (although for many
special classes of topological groupoids, this can in fact not happen).
Therefore we proceed in the following way.
We use the principal $(\Gg,\Gg')$-bundle $(P\times I,P'\times I)$
over $(\Gh\times I,\Gh'\times I)$, glue this bundle
with $(E,E')$ along the isomorphism
between $(E|_{\Gh\times \left\{1\right\}},E'|_{\Gh'\times\left\{1\right\}})$
and $(P\times \{0\},P'\times \{0\})$ and with
$(F,F')$ along the isomorphism
between $(F|_{\Gh\times \left\{0\right\}},F'|_{\Gh'\times\left\{0\right\}})$
and $(P\times \{1\},P'\times \{1\})$.
After reparametrization, we obtain
a principal $(\Gg,\Gg')$-bundle over $(\Gh\times I,\Gh'\times I)$ which gives
a homotopy between 
$(E|_{\Gh\times \left\{0\right\}},E'|_{\Gh'\times\left\{0\right\}})$ and
$(F|_{\Gh\times \left\{1\right\}},F'|_{\Gh'\times\left\{1\right\}})$.
\end{proof}

\begin{Remark} \rm
Observe that if the homotopies in the proof above are given by numerable
principal bundles $(E,E')$ and $(F,F')$, then the resulting concatenated homotopy
is again numerable.
\end{Remark}

\Para{Classifying space of a topological groupoid.}
Let $\Gg$ be a topological groupoid. Let us first recall the Milnor infinite join
construction of
the right $\Gg$-space $E\Gg$,
that is 
$$
E\Gg=\left\{(t_1g_1,t_2g_2,\ldots) | t_i\in \left[0,1\right],g_i\in \Gg_1,\sum
t_i=1,s(g_1)=s(g_2)=\cdots\right\}
$$
where only finitely many $t_i$ are non-zero
and we have a relation that $0g=0g'$ even if $g\neq g'$.
For further details see \cite[p.140]{Haefliger}. 

The space $E\Gg$ is a right $\Gg$-space along the map
$\bar{\epsilon}:E\Gg\to \Gg_0$,
$$
\bar{\epsilon}(t_1g_1,t_2g_2,\ldots)=s(g_i)
$$
with respect to the right $\Gg$ action
$$
(t_1g_1,t_2g_2,\ldots)g=(t_1(g_1g),t_2(g_2g),\ldots)
$$
for any $g\in\Gg$ with $t(g)=\bar{\epsilon}(t_1g_1,t_2g_2,\ldots)$.
The classifying space $B\Gg$ of the groupoid $\Gg$ is the space of orbits
if this action,
$$ B\Gg=E\Gg/\Gg .$$
It is known (see \cite{Dold,Haefliger,FibBund}) that
the quotient projection
$\pi:E\Gg\to B\Gg$ is the universal numerable principal $\Gg$-bundle over $B\Gg$,
thus the numerable
principal $\Gg$-bundles over a space $B$ are in bijective correspondence  (via
pullback) with homotopy classes of maps from $B$ to $B\Gg$.
    
We can generalize
the notion of the universal numerable principal bundle to
pairs of topological groupoids. Again, using Milnor's infinite join
construction, we get a bundle $(E\Gg,E\Gg')\to(B\Gg,B\Gg')$.
(The inclusion $i:\Gg'\to\Gg$ induces an inclusion $i:E\Gg'\to E\Gg$,
so we can view $E\Gg'$ as a subspace of $E\Gg$ in a natural way.)
We also have
the induced injective map $i:B\Gg'\to B\Gg$. Let us inspect the following
diagram:
$$
 \xymatrix{
 i(E\Gg')\ar[d]_-\pi\ar[r]&E\Gg'\ar[d]^-{\pi'}\\
 i(B\Gg')\ar[r]&B\Gg'
 }
 $$
The upper map is a homeomorphism and both vertical maps are quotient maps (as they
have local sections). Therefore, the lower map $i(B\Gg')\to B\Gg'$ is
continuous. This shows that $i:B\Gg'\to B\Gg$ is an embedding, so  we can view
$B\Gg'$ as a subspace of $B\Gg$.

Using basically the same arguments as in \cite[p.57]{FibBund} we see that
the homotopy classes of numerable principal $(\Gg,\Gg')$-bundles over topological
pair $(B,A)$ are in bijective correspondence (via pullback) with the homotopy
classes of maps $(B,A)\to(B\Gg,B\Gg')$.

All of the above can be in an obvious way generalized to $n$-tuples
$(\Gg,\Gg',\Gg'',\ldots)$ of topological groupoids.

\Para{Homotopy groups.}
Let $(\Gg,a_0)$ be a pointed topological groupoid and let $n\in \NN$.
We define $\pi_n(\Gg,a_0)$ to be the set of
homotopy classes of principal $(\Gg,a_0)$-bundles over $(I^n,\partial I^n)$,
$$ \pi_n(\Gg,a_0)=\left[\GPD_{\scriptscriptstyle\square}((I^n,\partial I^n),(\Gg,a_0))\right].$$
For $n\geq 1$, the set $\pi_n(\Gg,a_0)$ is in fact a group with respect
to the multiplication given by the concatenation.
Indeed, if $(P,\sigma)$ and $(P',\sigma')$ are two
principal $(\Gg,a_0)$-bundles over $(I^n,\partial I^n)$,
there is an uniquely determined isomorphism $h$ between the restrictions
$P|_{\left\{1\right\}\times I^{n-1}}$ and $P'|_{\left\{0\right\}\times I^{n-1}}$
which respects the sections $\sigma$ and $\sigma'$.
We glue the bundles $P$ and $P'$ along the isomorphism $h$ to obtain
a principal $(\Gg,a_0)$-bundle over $(I^n,\partial I^n)$, after reparametrization of
the base space.
With this concatenation operation the set $\pi_n(\Gg,a_0)$ becomes
a group, called the $n$-th homotopy group of the pointed topological groupoid $(\Gg,a_0)$.
Sometimes (for example, when the base point is clear from the context)
we write simply $\pi_n(\Gg,a_0)=\pi_n(\Gg)$.

Homotopy groups of pointed topological groupoids are generalizations of
the classical homotopy groups of topological spaces, since 
clearly we have $\pi_n(X,x_0)=\pi_n(X\rightrightarrows X,x_0)$.
In the same way as for homotopy groups of topological spaces, one can see that
the homotopy groups of pointed topological groupoids $\pi_n(\Gg,a_0)$ are
abelian for $n\geq 2$.   
  
Any principal $(\Gg,a_0)$-bundle $(P,p_0)$ over $(\Gh,b_0)$ induces a homomorphism (via
composition) $\pi_{n}(P,p_0):\pi_n(\Gh,b_0)\to\pi_n(\Gg,a_0)$. Furthermore,
this defines functors
$$ \pi_{0}:\GPD_{\circ}\to \textsf{Sets}_{\circ} $$
$$ \pi_{1}:\GPD_{\circ}\to \textsf{Grp} $$
and
$$ \pi_{n}:\GPD_{\circ}\to \textsf{Ab} $$
for $n\geq 2$,
where $\textsf{Sets}_{\circ}$, $\textsf{Grp}$ and $\textsf{Ab}$ stand
for the category of pointed sets, groups and abelian groups respectively.
By definition, all this functors are Morita invariants.
By the abuse of notation, we often write $\pi_{n}(P,p_0)=\pi_{n}(P)$.

One can easily check that the following theorem holds:
  
\begin{Theorem}\label{Theorem:th1}
Let $(P,p_0)$ and $(Q,q_0)$ be homotopic principal $(\Gg,a_0)$-bundles over
$(\Gh,b_0)$. Then $(P,p_0)$ and $(Q,q_0)$ induce the same homomorphism from
$\pi_n(\Gh,b_0)$ to $\pi_n(\Gg,a_0)$.    
\end{Theorem}

\begin{Example} \rm
(1)
Let $\textnormal{Pair}(X)=(X\times X\to X)$
be the pair groupoid over a pointed topological space $(X,x_0)$,
in which the source and the target
maps are the projections. This groupoid is weakly equivalent
to the space with only one point $x_0$, which means that
$\pi_n(\textnormal{Pair}(M))=0$ for any $n$.

(2)
Let $p:N\to M$ be a surjective submersion and $N\times_M N\rightrightarrows N$
the associated groupoid. Then the natural functor $\Phi$
from $N\times_M N\rightrightarrows N$ to $M\rightrightarrows M$ is a Morita
equivalence, which follows from \cite[p.128]{IntroFolandLie} and the fact that
the map $\left\langle \Phi \right\rangle\to M$ has local sections. This
gives us isomorphisms
$$
\pi_n(N\times_M N\rightrightarrows N,n_0)\cong\pi_n(M\rightrightarrows M,p(n_0))
\cong\pi_n(M,p(n_0))
$$
for any $n\in\NN$.

\end{Example}

\Para{Base point change.}
Let $\Gg$ be a topological groupoid
and suppose that $a_0,a_1\in \Gg_0$
are connected by a path in $\Gg_0$. Then, using identical argument as in the case of
homotopy groups of topological spaces, we get an isomorphism
between $\pi_n(\Gg,a_0)$ and $\pi_n(\Gg,a_1)$.

Now suppose that there is an arrow $g\in \Gg_1$ from $a_1\in \Gg_0$ to $a_0\in
\Gg_0$.  Then the principal $(\Gg,a_0)$-bundle $(\Gg_1,g)$ over
$(\Gg,a_1)$ (with $\pi=t:\Gg_1\to\Gg_0$ and $\epsilon=s:\Gg_1\to\Gg_0$)
is a Morita equivalence, so it induces an isomorphism between $\pi_n(G,a_0)$ and
$\pi_n(G,a_1)$.
  
It follows that if the topological groupoid $\Gg$ is $\Gg$-connected
(i.e. for every $a_0,a_1\in \Gg_0$ there is a $\Gg$-path from $a_0$ to $a_1$, see
\cite[p.28]{HilSkan}), then the groups $\pi_n(\Gg,a_0)$ and $\pi_n(\Gg,a_1)$
are isomorphic.
  
\Para{Homotopy groups of the classifying space.}  From the classification of
principal bundles we see that the homotopy groups $\pi_n(\Gg,a_0)$
of a pointed topological groupoid $(\Gg,a_0)$
are isomorphic to the homotopy groups of the
classifying space $B\Gg$.

\Para{Totally numerable principal bundles}
We recall from \cite{Dold} the definition of a halo. Let $X$ be a
topological space. A halo over a subset $B\subset X$ is a subset $U\subset X$
such that $B\subset U$ and there is a function $\tau:X\to \left[0,1\right]$
with $\tau|_B=1$ and $\textnormal{supp}(\tau)\subset U$.
  
Let $(\Gg,A)$ and $(\Gh,B)$ be marked topological groupoids and $(P,\sigma)$ a
numerable principal $(\Gg,A)$-bundle over $(\Gh,B)$. The bundle $(P,\sigma)$ is
totally numerable if the section $\sigma:B\to P$ can be extended to a halo around $B$.

We observe that any principal $(\Gg,a_0)$-bundle $(P,\sigma)$
over $(I^n,\partial I^n)$ is homotopic to a totally numerable principal
$(\Gg,a_0)$-bundle over $(I^n,\partial I^n)$.
Indeed, to see this, choose a relative homotopy
$H_t:(I^n\times I,\partial I^n\times I)\to(I^n,\partial I^n)$
from the identity to a map
which retracts an open neighbourhood of $\partial I^{n}$ in $I^n$
to $\partial I^n$. When we take the pullback of
$(P,\sigma)$ along the map $H$, we get a homotopy between the bundle
$(P,\sigma)$ and a totally numerable principal $(\Gg,a_0)$-bundle over
$(I^n,\partial I^n)$. 

We can in fact generalize the above statement about totally numerable bundles.
Let $P$ be a principal $\Gg$-bundle over $I^n$ and $\sigma:\partial I^n\to P$ a
section over $\partial I^n$ such that $\epsilon(\sigma(b))=a_0$ for every
$b\in\partial I^n$. The pair $(P,\sigma)$ may not be a principal $(\Gg,a_0)$-bundle
over $(I^n,\partial I^n)$ according to our definition, because we
the section $\sigma$ may not have local extensions. However, by using the same
argument and homotopy $H$ as above, we can see that such a
bundle $P$ with a section $\sigma$
is homotopic, via a homotopy with a suitable section,
to a totally numerable principal $(\Gg,a_0)$-bundle over
$(I^n,\partial I^n)$.
Similarly, any homotopy between principal $\Gg$-bundles over $I^n$
with sections over $\partial I^{n}$, equipped with a suitable section
over $\partial I^n\times I$,
can be transformed into a homotopy that is itself a numerable principal
$(\Gg,a_0)$-bundle over $(I^n\times I,\partial I^n\times I)$. 
It follows that the elements of the group $\pi_n(\Gg,a_0)$ can be
viewed as the homotopy classes
of principal $\Gg$-bundles $P$ over $I^n$, equipped with
a section $\sigma:\partial I^n\to P$ with $\epsilon(\sigma(b))=a_0$
for any $b\in\partial I^n$.

\section{Serre fibrations}

\textsl{
In Section~\ref{section:sec1} we described some basic  properties of homotopy
groups of  pointed topological groupoids. In this section, we first show that
any Morita map $P$ between pointed topological groupoids induces a long exact
sequence that links the homotopy groups of the pointed topological groupoids
and certain groups $\Sigma_n(P)$, which we describe explicitly and play the role of
homotopy groups of ``the homotopy fiber'' of the Morita map $P$.
Furthermore, we define what it means for a Morita map
from groupoid $\Gh$ to $\Gg$ to be a Serre
fibration. We show that if a Morita map $P$ between pointed
topological groupoids $(\Gh,b_0)$ and  $(\Gg,a_0)$ is a Serre fibration,
then the groups $\Sigma_n(P)$ can be identified as the homotopy groups
of a pointed topological groupoid, namely the fiber of $P$.}
\vspace{0.25cm}
    
We know that for an ordinary map between topological spaces, one has the
homotopy fiber of that map. Using this homotopy fiber, one can show that every
map between topological spaces fits into a long exact sequence of homotopy
groups (see \cite[p.407]{AT}).

We will now give a similar construction for  principal
$(\Gg,a_0)$-bundle $(P,p_0)$ over $(\Gh,b_0)$, where $(\Gg,a_0)$ and
$(\Gh,b_0)$ are given pointed topological groupoids.
Write $L^{n+1}\subset \partial I^{n+1}$ for the face of $I^{n+1}$
determined by the equation $t_{n+1}= 1$, where $t_{n+1}$ denotes the last
coordinate on $I^{n+1}=[0,1]^{n+1}\subset \RR^{n+1}$.
Denote by $J^{n+1}$ the union of all the remaining  faces of $I^{n+1}$,
this $J^{n+1}$ is the closure of $\partial I^{n+1}\setminus L^{n+1}$ in $I^{n+1}$.

Let $S_n(P)$ be the set of all
triples $(\alpha,\beta,h)$,
where $\alpha$ is a principal $(\Gh,b_0)$-bundle over $(I^n,\partial I^n)$,
$\beta$ is a principal $(\Gg,a_0)$-bundle over $(I^{n+1},J^{n+1})$
and $h$ is an isomorphism from $\alpha\otimes P$ to $\beta|_{L^{n+1}}$. 
Note that here we identified $L^{n+1}$ with $I^n$ in the canonical way,
and that the sections are implicit in the
definition: for instance, the bundle $\alpha$ is  actually a bundle
$(\alpha,\rho)$, where $\rho$ is a section over $\partial I^n$ and $h$ is an
isomorphism that preserves the sections.

An isomorphism between triples
$(\alpha,\beta,h),(\alpha',\beta',h')\in S_n(P)$
is a pair of isomorphisms
$q:\alpha\to\alpha'$, $Q:\beta\to\beta'$ such that the diagram
$$
\xymatrix{
\alpha\otimes P\ar[d]_-h\ar[r]^-{q\otimes \textnormal{id}} & \alpha'\otimes P\ar[d]^-{h'} \\
\beta|_{L^{n+1}}\ar[r]_-{Q|_{L^{n+1}}} & \beta'|_{L^{n+1}}
}
$$
commutes.
We say that the triples $(\alpha,\beta,h)$ and $(\alpha',\beta',h')$
are homotopic if there is a triple $(A,B,H)$,
where $A$ is a principal $(\Gh,b_0)$-bundle over $(I^n\times I,\partial I^n\times I)$,
$B$ is principal $(\Gg,a_0)$-bundle over $(I^{n+1}\times I,J^{n+1}\times
I$ and $H$ is an isomorphism from $A\otimes P$ to $B|_{{L^{n+1}}\times I}$
such that
$$(A|_{(I^n\times\left\{0\right\},\partial
I^n\times\left\{0\right\})}, B|_{(I^n\times\left\{0\right\},\partial
I^n\times\left\{0\right\})}, H|_{(I^n\times\left\{0\right\},\partial
I^n\times\left\{0\right\})})$$
is isomorphic to $(\alpha,\beta,h)$ and
$$(A|_{(I^n\times\left\{1\right\},\partial
I^n\times\left\{1\right\})},  B|_{(I^n\times\left\{1\right\},\partial
I^n\times\left\{1\right\})}, H|_{(I^n\times\left\{1\right\},
\partial I^n\times\left\{1\right\})})$$
is isomorphic to $(\alpha',\beta',h')$. Using
similar arguments as in the proof of Proposition~\ref{prop:homotopy}
we see that  this gives us an equivalence relation on the set $S_n(P)$.
We denote the set of homotopy classes of triples
in $S_n(P)$ by
$$\Sigma_n(P).$$ 
Concatenation of triples is defined in the same manner as for the
homotopy classes of principal $(\Gg,a_0)$-bundles over $(I^n,\partial I^n)$
and induces a group structure on $\Sigma_n(P)$.

Notice that, similarly to the case of homotopy groups, every triple 
$(\alpha,\beta,h)$ in $S_n(P)$ is homotopic to a numerable
triple, that is, to a triple $(\alpha',\beta',h')$ such that 
$\alpha'$ and $\beta'$ are numerable principal bundles. Furthermore,
as in the case of homotopy groups, we can safely ignore the condition
on local extendability of sections over $\partial I^n$, respectively $J^{n+1}$.

\begin{Theorem}\label{Theorem:LES_general}
Let $(P,p_0)$ be a principal $(\Gg,a_0)$-bundle over $(\Gh,b_0)$. Then there is
a natural long exact sequence
$$
\ldots\to \Sigma_n(P)\to\pi_n(\Gh,b_0)\stackrel{\pi_n(P)}{\to}\pi_n(\Gg,a_0)\to
\Sigma_{n-1}(P)\to\pi_{n-1}(\Gh,b_0)\to\ldots $$
\end{Theorem}

\begin{proof}
Let $Y_n(P)$ denote the set of triples $(\alpha,\beta,h)$, where
$\alpha$ is principal $(\Gh,b_0)$-bundle over
$(I^n,\partial I^n)$,
$\beta$ is principal $(\Gg,a_0)$-bundle over $(I^n\times I,\partial I^n\times I)$
and $h$ is an isomorphism $h:\alpha\otimes P\to\beta|_{L^{n+1}}$.
As in the case of $S_n(P)$, we have homotopies of such triples. We denote the
set of homotopy classes of triples in $Y_n(P)$ by $\Upsilon_n(P)$.
Concatenation of triples in $Y_n(P)$ induces a group structure on
$\Upsilon_n(P)$.

Let us first check that there is an isomorphism 
$\varphi:\Upsilon_n(P)\to\pi_n(\Gh,b_0)$ given by
$\varphi(\alpha,\beta,h)=\alpha$.
Indeed, for the inverse we take
$\varphi^{-1}(\alpha)=(\alpha,(\alpha\otimes P)\times I,\textnormal{id})$,
where $(\alpha\otimes P)\times I$ denotes the pullback of
$\alpha\otimes P$ along the projection
$(I^{n}\times I,\partial I^{n}\times I) \to (I^{n},\partial I^{n})$.
Both maps are well
defined (on the homotopy classes of triples) and the
composition $\varphi\circ\varphi^{-1}$ is clearly the identity.
We have to check the surjectivity of
$\varphi^{-1}\circ\varphi$ is also the identity.
To see this, we need to show that the triples
$(\alpha,\beta,h)$ and $(\alpha,\alpha\otimes P\times I,\textnormal{id})$
are homotopic in $Y_n(P)$.
Indeed, first we have the isomorphism $(id_\alpha,h\times \textnormal{id})$
from $(\alpha,\alpha\otimes P\times I,\textnormal{id})$ to
$(\alpha,\beta|_{L^{n+1}}\times I,h)$,
and from here we have the homotopy to
$(\alpha,\beta,h)$
of the form $(\alpha\times I,B,h\times I)$,
where $B$ is the pullback of $\beta$ along the map
$I^n\times I\to I^n$,
$((t_1,\ldots,t_{n-1},t_n),t)\mapsto(t_1,\ldots,t_{n-1},(1-t)t_n+t)$,
and $h\times I$ denotes the isomorphism induced by $h$ on the corresponding pullback.

Now we have to check that the sequence 
$$
\ldots\to\Sigma_n(P)\to\Upsilon_n(P)\to\pi_n(\Gg,a_0)
\to\Sigma_{n-1}(P)\to\Upsilon_{n-1}(P)\to\ldots
$$
is exact. The map $\Sigma_n(P)\to\Upsilon_n(P)$ is induced by the inclusion
$S_n(P)\to Y_n(P)$ (which restricts the implicit sections). 
The map $\Upsilon_n(P)\to\pi_n(\Gg,a_0)$ maps (the homotopy class of)
$(\alpha,\beta,h)$ to $\beta|_{I^{n}\times \{0\}}$.
The map $\pi_n(\Gg,a_0)\to\Sigma_{n-1}(P)$ maps $\beta$ to
$\Delta(\beta)=(\bar{b_0},\beta,\iota)$, where $\iota$ is the uniquely determined isomorphism of
bundles $\bar{b_0}$ (a bundle with global section) and $\beta|_{L^{n}}$ that
maps the global section of $\bar{b_0}$ to the global section of
$\beta|_{L^{n}}$.

(i) Exactness at $\Upsilon_n(P)$: The composition
$\Sigma_n(P)\to\Upsilon_n(P)\to\pi_n(\Gg,a_0)$ is zero, since the bundle
$\beta|_{I^{n}\times\{0\}}$ is trivial $(\Gg,a_0)$-bundle. 
On the other hand, if the image of the triple
$(\alpha,\beta,h)$, which equals $\beta|_{I^{n}\times\{0\}}$,
is homotopic to the trivial
$(\Gg,a_0)$-bundle, than  we just concatenate this homotopy with
$\beta$ to obtain a triple in $\Sigma_n(P)$ which maps to $(\alpha,\beta,h)$.

(ii)
Exactness at $\pi_n(\Gg,a_0)$: The composition
$\Upsilon_n(P)\to\pi_n(\Gg,a_0)\to\Sigma_{n-1}(P)$ is zero because
$(\alpha,\beta,h)\in\Upsilon_n(P)$ can be viewed, after deformation of the base space,
as a homotopy from $\Delta(\beta|_{I^{n}\times \{0\}})$ to trivial triple in $S_{n-1}(P)$.
Furthermore, if $\beta$ represents an element in $\pi_n(\Gg,a_0)$
such that $\Delta(\beta)$ is homotopic to the trivial triple in $S_{n-1}(P)$,
then this homotopy
can be viewed, after deformation of the base space, as an en element of $Y_{n}(P)$
which maps to the homotopy class of $\beta$ in $\pi_n(\Gg,a_0)$.
  
(iii) Exactness at $\Sigma_{n-1}(P)$: To see that the composition
$\pi_n(\Gg,a_0)\to\Sigma_{n-1}(P)\to\Upsilon_{n-1}(P)$ is trivial,
observe that the triple $\Delta(\beta)$ is homotopic to the trivial triple
in $Y_{n-1}(P)$ precisely because the implicit section is restricted.
On the other hand, if $(\alpha,\beta,h)$ is a triple in $S_{n-1}(P)$
which is homotopic to the trivial triple in $Y_{n-1}(P)$, then
homotopy, after deformation of the base space, represents
an element $\pi_n(G,a_0)$ which maps to the homotopy class
of $(\alpha,\beta,h)$.
\end{proof}

\begin{Definition} \rm
Let $\Gh$ and $\Gg$ be  topological groupoids and let
$P$ be  a principal $\Gg$-bundle over $\Gh$. The bundle
$P$ is a Serre fibration if for every triple $(\alpha,\beta,h)$,  where
$\alpha$ is a principal $\Gh$-bundle over $I^n$,
$\beta$ is a principal $\Gg$-bundle
over $I^{n+1}$ and
$h$ an isomorphism from $\alpha\otimes P$ to $\beta|_{I^n\times\left\{0\right\}}$,
there is a triple $(A,B,H)$ with $A$
a principal $\Gh$-bundle  over $I^{n+1}$, $B$ a principal $\Gg$-bundle  over
$I^{n+1}$ and $H$ an isomorphism from $A\otimes P$ to $B$ of principal
$\Gg$-bundles over $I^{n+1}$, such that the triples $(\alpha,\beta,h)$ and
$(A|_{I^n\times\left\{0\right\}},B,H|_{I^n\times\left\{0\right\}})$ are
isomorphic.
\end{Definition}

\begin{Remark} \rm
The notion of the isomorphism between triples $(\alpha,\beta,h)$,
as above is obvious, similar to the one used in the
definition of homotopy between the triples in $S_{n}(P)$ - the only difference
is that here the topological groupoids and principal bundles are not
assumed to be pointed.
Clearly, the notion of a Serre fibration is a well defined property
of a Morita map between topological groupoid, although its definition
is essentially intrinsic in the Morita bicategory of topological groupoids.
We see that the notion of Serre fibration in $\GPD$ is a generalization
of the notion of Serre fibration in the category of topological spaces, as the
above definition can be presented in the diagram
$$
\xymatrix{
I^n\ar[r]^-\alpha\ar[d]& \Gh\ar[d]^-{P}\\
I^{n+1}\ar[r]_-\beta\ar[ru]^-{A} & \Gg
}
$$
in the Morita bicategory of topological groupoids.
Note that we obtain an equivalent definition of a Serre fibration if
we replace $I^n=I^{n}\times\{0\}$ 
with $J^{n+1}$, or by $K^{n+1}=(I^{n}\times\{1\})\cup
(\partial I^{n}\times I)\subset I^{n+1}$, etc.
\end{Remark}

\begin{Proposition}\label{SerreProperties}
Let $P$ be a principal $\Gg$-bundle over $\Gh$ and
$Q$ a principal $\Gh$-bundle over $\Gk$.
\begin{enumerate}
\item [(i)] If $P$ is a Morita equivalence, then it is a Serre fibration.
\item [(ii)] If $P$ and $Q$ are both Serre fibrations, then $Q\otimes P$
             is also a Serre fibration.
\end{enumerate}
\end{Proposition}

\begin{proof}
It is straightforward to check both assertions.
To check (i), one uses the fact that for if $P$ is a Morita
equivalence, then the inverse of $P$ in the Morita category
can be represented by the principal $\Gh$-bundle $P^{-1}$ over $\Gg$,
which equals $P$ as the topological spaces, but has actions transposed.
In this way, there are in fact natural isomorphisms
$P\otimes P^{-1}\cong\Gh$ and $P^{-1}\otimes P\cong \Gg$, which
are to be used in the argument.
\end{proof}

\begin{Proposition}\label{prop:Serrefunctor}
Let $\phi:\Gh\to\Gg$ be a functor between topological groupoids
such that $\phi:\Gh_0\to\Gg_0$ is a Serre fibration and
$(\phi,s):\Gh_1\to\Gg_1\times_{\Gg_0}\Gh_0$
is a surjective Serre fibration. Then the associated
principal $\Gg$-bundle $\langle\phi\rangle$ over $\Gh$
is a Serre fibration.
\end{Proposition}

Before we give the proof,
let us first recall the notion of a $\Gg$-cocycle of a
principal $\Gg$-bundle $P$ over a space $B$. The bundle $P$ has sections
$\{\sigma_i\}_{i\in\Lambda}$ over an open covering
$\{U_i\}_{i\in\Lambda}$ of $B$. Denote 
$g_{ij}(b)=\vartheta(\sigma_i(b),\sigma_j(b))$ for
$b\in U_i\cap U_j$, where $\vartheta$ is the translation function of
the bundle $P$. Write $f_i=\epsilon\circ\sigma_i:U_i\to\Gg_0$.
Note that
$s(g_{ij}(b))=f_j(b)$
$t(g_{ij}(b))=f_i(b)$ for any $b\in U_i\cap U_j$.
Furthermore,
$g_{ii}(b)=1_{f_i(b)}$ for $b\in U_i$ and
$g_{ij}(b)g_{jk}(b)=g_{ik}(b)$ for
$b\in U_i\cap U_j \cap U_k$.
We say that a family of functions $\{f_{i},g_{ij}\}$
satisfying the above conditions is a $\Gg$-cocycle on $B$.
Any $\Gg$-cocycle on $B$ on the other hand
determines a principal $\Gg$-bundle over $B$ \cite{HilSkan}. 

\begin{proof}[Proof of Proposition~\ref{prop:Serrefunctor}]
The maps from Proposition~\ref{prop:Serrefunctor} fit into the diagram:
$$
\xymatrix{
\Gh_1\ar[r]^-{(\phi,s)}\ar[dr]_-\phi&\Gg_1\times_{\Gg_0}\Gh_0\ar[d]_-{pr_1}\ar[r]^-{pr_2}
   &\Gh_0\ar[d]^-\phi \\
   & \Gg_1\ar[r]^-s & \Gg_0  
}
$$
It follows that $\phi:\Gh_1\to\Gg_1$ is a surjective Serre fibration, because
the maps $\phi:\Gh_0\to\Gg_0$ and
$(\phi,s):\Gh_1\to\Gg_1\times_{\Gg_0}\Gh_0$
are both surjective Serre fibrations.

We have to check the Serre fibration property for the principal bundle
$\langle\phi\rangle$.
Let $(\alpha,\beta,h)$ be a triple,
where $\alpha$ is a principal $\Gh$-bundle over $I^n$,
$\beta$ is a principal $\Gg$-bundle over $I^{n+1}$ and
$h$ is an isomorphism from $\alpha\otimes\langle\phi\rangle$ to $\beta|_{I^n\times\{0\}}$. 
We represent both bundles $\alpha$ and $\beta$ with cocycles,
by dissecting the cube $I^n$
into a family of small cubes. More precisely, we choose a large natural number
$N$ adn a small positive number $\epsilon$,
take $C_{i}=(\frac{i-1}{N}-\epsilon,\frac{i}{N}+\epsilon)\cap I$ for $i=1,\ldots N$
and
$$ C_{\mu}=C_{\mu_1}\times C_{\mu_2}\times\cdots\times C_{\mu_n}\subset I^n$$
for any multi-index
$\mu=(\mu_1,\mu_2,\ldots,\mu_n)$,
$\mu_1,\mu_2,\ldots,\mu_n=1,\ldots N$.
Then $\{C_\mu\}$ is a finite open cover of $I^n$.
If we replace the intervals $C_{i}$ in this construction with slightly
smaller closed intervals $D_{i}=[\frac{i-1}{N}-\frac{\epsilon}{2},
\frac{i}{N}+\frac{\epsilon}{2}]\cap I$,
we obtain a finite closed cover $\{D_\mu\}$ of $I^n$.
We do this analogously in the dimension $n+1$, obtaining
a finite open cover $C'_{\mu'}$ and a finite closed cover $D'_{\mu'}$ of $I^{n+1}$.
For $N$ large enough, we can represent
$\alpha$ and $\beta$ with cocycles
$\{f_\mu,h_{\mu\nu}\}$ and $\{F_{\mu'},g_{\mu'\nu'}\}$
on open covers $\{C_\mu\}$ and $\{C'_{\mu'}\}$ respectively.
We can restrict these cocycles to the closed covers
$\{D_\mu\}$ respectively $\{D'_{\mu'}\}$, obtaining so called
``closed'' cocycles which equally well represent the principal bundles
(we will use the same notation for this restrictions).
The principal bundle  $\alpha\otimes\langle\phi\rangle$ is then
given by the cocycle $\{\phi\circ f_\mu,\phi\circ h_{\mu\nu}\}$.
The isomorphism $h$ from $\alpha\otimes\langle\phi\rangle$ to
$\beta|_{I^n\times\{0\}}$ is, in terms of the cocycles,
given by a family of functions $r_\mu:C_\mu\to\Gg_1$.

Now we will lift the closed $\Gg$-cocycle $\{F_{\mu'},g_{\mu'\nu'}\}$
to a $\Gh$-cocycle along $\phi$.
First observe that we can lift the functions $r_\mu:D_\mu\to\Gg_1$
along $\phi$ to functions
$\widetilde{r}_\mu:D_\mu\to\Gh_1$
because $(\phi,s):\Gh_1\to\Gg_1\times_{\Gg_0}\Gh_0$ is a surjective Serre
fibration. Since $\phi:\Gh_0\to\Gg_0$ is a Serre fibration,
we can lift $F_{(1,1,\ldots,1)}$ along $\phi$ to
$\widetilde{F}_{(1,1,\ldots,1)}:D'_{(1,1,\ldots,1)}\to\Gh_0$ such
that $t\circ \widetilde{r}_{(1,1,\ldots,1)}=
\widetilde{F}_{(1,1,\ldots,1)}|_{D_{(1,1,\ldots,1)}}$. 
From the elements of the cocycle already lifted we calculate the
initial lift of the element $g_{(1,1,\ldots,1)(1,2,\ldots,1)}$, and because
$(\phi,s):\Gh_1\to\Gg_1\times_{\Gg_0}\Gh_0$ is a Serre fibration we can lift the
entire function $g_{(1,1,\ldots,1)(1,2,\ldots,1)}$ to a function
$\widetilde{g}_{(1,1,\ldots,1)(1,2,\ldots,1)}:
D'_{(1,1,\ldots,1)}\cap D'_{(1,2,\ldots,1)}\to\Gh_1$.
Now the elements of the cocycle already
lifted determine the initial lift of the functions
$F_{(1,2,\ldots,1)}$, which can be lifted because
$\phi:\Gh_0\to\Gg_0$ is a Serre fibration. Proceeding in this way,
we lift the entire $\Gg$-cocycle $\{F_{\mu'},g_{\mu'\nu'}\}$,
and obtain a closed $\Gh$-cocycle. Restricting this cocycle to
the interiors of their domains, we obtain an open $\Gh$-cocycle
representing the desired principal $\Gh$-bundle over $I^{n+1}$.
\end{proof}

Let
$(P,p_0)$ be a principal $(\Gh,b_0)$-bundle over $(\Gg,a_0)$.
We see that the $\Gh$-action on $P$ restricts to
$\epsilon^{-1}(a_0)$. Therefore, we
get the translation groupoid
$$\Gh\ltimes\epsilon^{-1}(a_0)
=(\Gh\times_{\Gh_0}\epsilon^{-1}(a_0)\rightrightarrows\epsilon^{-1}(a_0)),$$
which we call the fiber of $P$ over $a_{0}$.
The groupoid $\Gh\ltimes\epsilon^{-1}(a_0)$ is also
pointed with  $p_0\in\epsilon^{-1}(a_0)$.

\begin{Lemma}\label{Lemma:H_eps}
Let $(X,\sigma)$ be principal $(\Gh,b_0)$-bundle over $(I^n,\partial I^n)$
such that the principal $(\Gg,a_0)$-bundle 
$(X\otimes P,\sigma\otimes p_0)$ over $(I^n,\partial I^n)$
is trivial.
Then there exists a right
$\Gh\ltimes\epsilon^{-1}(a_0)$-action on $X$
such that $(X,\sigma)$ is a principal
$(\Gh\ltimes\epsilon^{-1}(a_0),p_0)$-bundle over $(I^n,\partial I^n)$.
\end{Lemma}

\begin{proof}
Since the $(\Gg,a_0)$-bundle $(X\otimes P,\sigma\otimes p_0)$ is a trivial,
it has a global section $\widetilde{\sigma}$ that extends $\sigma\otimes p_0$
maps to $a_0$ via $X\otimes P\to \Gg_{0}$.
For any $x\in X$ we have
$\widetilde{\sigma}(\pi(x))=x\otimes \alpha(x)$ for an uniquely determined
$\alpha(x)\in\epsilon^{-1}(a_0)$, this gives us a map
$\alpha:X\to\epsilon^{-1}(a_0)$. Note that
$\alpha(\sigma(b))=p_0$ for any $b\in\partial I^n$. 

The right action of $\Gh\ltimes\epsilon^{-1}(a_0)$ on $X$ along 
$\alpha$ is given by $x(h,p)=xh$ for any $x\in X$ and
$(h,p)\in \Gh\times_{\Gh_0}\epsilon^{-1}(a_0)$ with
$\alpha(x)=hp$ (note that, by applying $\pi$, that this equation
implies $\epsilon(x)=t(h)$).
One can check that
$(X,\sigma)$ is a principal $(\Gh\ltimes\epsilon^{-1}(a_0),p_0)$-bundle.
Indeed, if $x(h,p)=x(h',p')$, then $xh=xh'$ and hence $h=h'$,
while $\alpha(x)=hp=h'p'$ implies $p=p'$. This means that the
$\Gh\ltimes\epsilon^{-1}(a_0)$-action on $X$ is free.
To see that the action is transitive along the fibers of $\pi:X\to I^n$,
let $x,x'\in X$ with $\pi(x)=\pi(x')$. We can choose $h\in\Gh_{1}$ such that
$xh=x'$, and derive $x'=x(h,h^{-1}\alpha(x))$.
\end{proof}

\begin{Theorem}\label{Theorem:LES}
Let $(P,p_0)$ be a principal $(\Gh,b_0)$-bundle over $(\Gg,a_0)$ such that
$P$ is a Serre fibration.
Then there exist a natural long exact sequence
$$
\ldots\to\pi_n(\Gh\ltimes\epsilon^{-1}(a_0))\stackrel{\pi_n(\textnormal{pr}_1)}{\to}\pi_n(\Gh)
\stackrel{\pi_{n}(P)}{\to}\pi_n(\Gg)\to\pi_{n-1}(\Gh\ltimes\epsilon^{-1}(a_0))\to\ldots 
$$ 
\end{Theorem}

\begin{proof}
By Theorem \ref{Theorem:LES_general}, it is sufficient to
prove that the groups $\Sigma_n(P)$ are isomorphic to the
groups $\pi_n(\Gh\ltimes\epsilon^{-1}(a_0),p_0)$. 

First, we define a map $\psi:\Sigma_n(P)\to \pi_n(\Gh\ltimes\epsilon^{-1}(a_0),p_0)$,
as follows: for a triple $(\alpha,\beta,h)\in S_{n}(P)$,
we extend $\alpha$ to $K^{n+1}$ with a trivial
$(\Gh,b_0)$-bundle over
$(\partial I^n\times I,\partial I^n\times I)$
and extend the isomorphism $h$ to $K^{n+1}$
so that it becomes an isomorphism of bundles over
$(K^{n+1},\partial I^n\times I)$. Then the Serre fibration
property of $P$ gives us a triple $(A,B,H)$, and we use
Lemma~\ref{Lemma:H_eps} on $A|_{I^n\times\left\{0\right\}}$ to
get a principal
$(\Gh\ltimes\epsilon^{-1}(a_0),p_0)$-bundle
$\psi(\alpha,\beta,h)$
over $(I^n,\partial I^n)$.
Using again the Serre fibration property of $P$, one can
see that this map is well defined on $\Sigma_n(P)$, i.e. depends only
on the homotopy class of the triple.

To describe the inverse map
$\psi^{-1}$, let $(Q,\rho)$ be a principal
$(\Gh\ltimes\epsilon^{-1}(a_0),p_0)$-bundle over $(I^n,\partial I^n)$.
Denote by $\textnormal{pr}:\Gh\ltimes\epsilon^{-1}(a_0)\to\Gh$ the natural projection,
which equals $\pi$ on objects, and write $\bar{a_0}=I^{n+1}\times t^{-1}(a_{0})$ for the trivial
$(\Gg,a_0)$-bundle over $I^{n+1}$. 
We set $\psi^{-1}(Q,\rho)=(Q\otimes\langle\textnormal{pr} \rangle,\bar{a}_0,\iota)$,
where $\iota:Q\otimes\langle \textnormal{pr}\rangle\otimes P\to\bar{a_0}|_{L^{n+1}}$
is an isomorphism given by
$$ \iota(q\otimes(p,h)\otimes p')= (\pi(q),\vartheta(p,hp')). $$
This gives us a map $\psi^{-1}:\pi_n(\Gh\ltimes\epsilon^{-1}(a_0),p_0)\to \Sigma_n(P)$.

To check that $\psi\circ\psi^{-1}$ is the identity, observe that in the construction of 
the homotopy class $\psi(\psi^{-1}(Q,\rho))$ one may take $A$ to be the pullback
of $Q\otimes\langle\textnormal{pr} \rangle$ along the projection $I^{n}\times I\to I^{n}$,
and it is then sufficient to note that there is an isomorphism
of $(\Gh\ltimes\epsilon^{-1}(a_0),p_0)$-bundles between
$Q\otimes\langle\textnormal{pr} \rangle$ (viewed as a principal
$\Gh\ltimes\epsilon^{-1}(a_0),p_0)$-bundle by Lemma~\ref{Lemma:H_eps})
and $Q$, which maps $q\otimes(p,h)$ to $q(h,h^{-1}p)$.

Finally, we have to check that $\psi^{-1}\circ\psi$ is the identity as well.
Let $(\alpha,\beta,h)\in S_{n}(P)$ and construct $\psi(\alpha,\beta,h)$
as above. It is sufficient to observe that the principal $(\Gh,b_0)$-bundles
$\psi(\alpha,\beta,h)\otimes \langle\textnormal{pr} \rangle$
and $A|_{I^n\times\left\{0\right\}}$ are naturally isomorphic
(the isomorphisms maps $a\otimes (p,h)$ to $ah$).
\end{proof}

\begin{Example}\rm
Let $\phi:(\Gh,b_0)\to (\Gg,a_0)$ be a functor between pointed topological groupoids
such that the associated principal $\Gg$-bundle $\langle\phi\rangle$
over $\Gh$ is a Serre fibration. Then there exists a natural long exact sequence
as in Theorem \ref{Theorem:LES} in which the fiber
$\Gh\ltimes \epsilon^{-1}(a_0)$ of $\langle\phi\rangle$ equals
the translation groupoid $\Gh\ltimes(\Gh_0\times_{\Gg_0}s^{-1}(a_0))$.
\end{Example}

\section{Serre groupoids}\label{sec:serre}

\textsl{
In this section we introduce a special class of topological groupoids 
called Serre groupoids. We show that the calculation of homotopy
groups of Serre groupoids is particularly simple.
Examples will show that there are many topological groupoids that are 
Morita equivalent to Serre groupoids. }
\vspace{0.25cm}

\begin{Definition}  \rm
A Serre groupoid is a topological groupoid $\Gg$ 
for which the source map $s:\Gg_{1}\to\Gg_{0}$ is a Serre fibration.
\end{Definition}

\begin{Remark} \rm
If $\Gg$ is a Serre groupoid, then the target map $t:\Gg_1\to \Gg_0$ is also a Serre
fibration.
\end{Remark}

\begin{Proposition}\label{Proposition:globsec}
Let $\Gg$ be a Serre groupoid and $P$ a 
principal $\Gg$-bundle over $\Gh$.
Then the map $\pi:P\to \Gh_{0}$ is a Serre fibration.
\end{Proposition}

\begin{proof}
The bundle $\pi:P\to \Gh_0$ has local sections over an open covering
$\{U_\lambda\}_{\lambda\in\Lambda}$. For any $\lambda$ we have the
pullback diagram
$$
\xymatrix{
P|_{U_\lambda}\ar[r]\ar[d]_-{\pi|_{U_{\lambda}}}&\Gg_1\ar[d]^-t\\
U_\lambda\ar[r]& \Gg_0.
}
$$
The map $\pi|_{U_\lambda}$ is a Serre fibration, because it is
a pullback of the Serre fibration $t$.
Thus $\pi:P\to \Gh_0$ is a Serre fibration locally over an open
covering of $\Gh_0$, which yields that it is itself a Serre fibration.
\end{proof}

\begin{Proposition}\label{prop:Serrefunctorbundle}
Let $\phi:\Gh\to \Gg$ be a continuous functor
between Serre groupoids which is a Serre fibration on objects.
Then the principal bundle $\left\langle \phi \right\rangle$
associated to $\phi$ is a Serre fibration.
\end{Proposition}

\begin{proof}
Let $\alpha$ be a principal $\Gh$-bundle over $I^n$,
$\beta$ a principal $\Gg$-bundle over $I^n\times I$ and
$h:\alpha\otimes \left\langle \phi \right\rangle\to \beta|_{I^n\times\{0\}}$
an isomorphisms. Both bundles $\alpha$ and $\beta$ have
global sections, since their projections are Serre fibrations
by Proposition~\ref{Proposition:globsec}. This means that
we can view $\alpha$ as the pullback along a map
$\alpha':I^n\to P$ and $\beta$ as a pullback
along a map $\beta':I^n\times I\to \Gg_0$.
Furthermore, there exists a natural isomorphisms of functors
$\phi_{0}\circ\alpha':I^{n}\to \Gg$ and
$\beta'|_{I^{n}\times \{0\}}\to\Gg$, given by a map
$w:I^{n}\to\Gg_1$.
Since the target map $\Gg_1 \to\Gg_0$ is a Serre fibration,
we can extend the map $w$ to $W:I^{n+1}\to\Gg_1$ such that
$t\circ W=\beta'$. Now $\beta''=s\circ W$ also represents
the bundle $\beta$, and
$\phi_{0}\circ\alpha'=\beta''|_{I^{n}\times \{0\}}$.
Finally, since $\phi_0$ is a Serre fibration,
we can extend $\alpha'$ to a map $A:I^{n+1}\to\Gh_0$ such that
$\phi_0 \circ A=\beta''$.
\end{proof}

\begin{Example}\rm
Let $\phi:H\to G$ be a continuous homomorphism between topological groups.
The topological groupoids $(H\rightrightarrows \ast)$ and
$(G\rightrightarrows \ast)$ representing $H$ and $G$ are clearly Serre groupoids,
and $\phi$ is a functor between these two groupoids which is a Serre fibration
(and in fact the identity) on objects.  By Proposition
\ref{prop:Serrefunctorbundle} it follows that the associated
principal bundle $\langle\phi\rangle$, the total space of which equals $G$,
is a Serre fibration. Theorem \ref{Theorem:LES} then gives us a long exact sequence
$$
\ldots\to
\pi_n(H\ltimes G)\stackrel{\pi_n(\textnormal{pr}_1)}{\to}
\pi_n(H\rightrightarrows\ast)\stackrel{\pi_{n}(\phi)}{\to}
\pi_n(G\rightrightarrows\ast)\to
\pi_{n-1}(H\ltimes G)\to\ldots 
$$ 
\end{Example}

\begin{Proposition}
Let $\Gh$ and $\Gg$ be Serre groupoids and  $P$ a principal $\Gg$-bundle over
$\Gh$ such that the map $\epsilon:P\to \Gg_0$ is a Serre fibration. Then $P$
is a Serre fibration.
\end{Proposition}

\begin{proof}
Recall that the projection $\textnormal{pr}_{1}:\Gh\ltimes P\rtimes \Gg\to\Gh$
is a weak equivalence(\cite{PoissonGeo})
and that $\langle\textnormal{pr}_{1}\rangle\otimes P\cong \langle\textnormal{pr}_{3}\rangle$
Proposition \ref{prop:Serrefunctorbundle} implies that
$\langle\textnormal{pr}_{3}\rangle$ is a Serre fibration.
Now it follows from Proposition \ref{SerreProperties} that $P$
is a Serre fibration as well.
\end{proof}

\begin{Example}
Let $\Gg$ be a Serre groupoid, acting on a space $X$.
Then the associated translation groupoid is also a Serre groupoid.
\end{Example}

The next theorem gives a method for calculating the homotopy groups of a Serre
groupoid.

\begin{Theorem}\label{Thm:LES_serreGPD}
Let $(\Gg,a_0)$ be a pointed Serre groupoid.
Then there is a natural long exact sequence
$$
\ldots\to\pi_n(s^{-1}(a_0),1_{a_0})\stackrel{\pi_n(t)}{\to}\pi_n(\Gg_0,a_0)
\to\pi_n(\Gg,a_0)\to\pi_{n-1}(s^{-1}(a_0),1_{a_0})\to\ldots
$$
\end{Theorem}

\begin{proof}
We have the obvious functor from the space $\Gg_0$
to the groupoid $\Gg_1$. The principal $\Gg$-bundle over $\Gg_0$
associated to this functor is identity on objects, hence it is
a Serre fibration by Proposition \ref{prop:Serrefunctorbundle}.
We can therefore apply Theorem \ref{Theorem:LES}.
\end{proof}

We will now use our results  to calculate homotopy
groups of some topological groupoids.

\begin{Example} \rm
(1)
A unit groupoid $X\rightrightarrows X$ over topological space $X$ is a Serre
groupoid, with $s^{-1}(x_0)=\{1_{x_0}\}$. The long exact sequence
of Theorem \ref{Thm:LES_serreGPD} gives us the already mentioned isomorphisms
$\pi_n(X\rightrightarrows X)\cong\pi_n(X)$.

(2)
Let $\textnormal{Pair}(X)$ be a pair groupoid over pointed space $X$.
It is also a Serre groupoid, because the source map is simply a projection.
In particular we have $s^{-1}(x_0)=X$.

(3)
A topological group $G$ is a topological groupoid $(G\rightrightarrows \ast)$
with one object, and it is also a Serre groupoid. The long exact sequence
in this case gives us the known result that
$$
\pi_n(G\rightrightarrows\ast)\cong\pi_{n-1}(G)\cong\pi_{n}(BG),
$$
where $BG$ is a classifying space of group $G$.

(4)
Let $\Gg$ be a Serre groupoid and $(X,x_0)$ a pointed
right $\Gg$-space. Then the groupoid
$X\rtimes \Gg$ is a pointed Serre groupoid, and we get a long exact sequence
$$
\ldots\to\pi_n(s^{-1}(\epsilon(x_0)))\to\pi_n(X)\to\pi_n(X\rtimes \Gg)\to
\pi_{n-1}(s^{-1}(\epsilon(x_0)))\to\ldots
$$
In case $\Gg$ is also \'{e}tale and $X$ is connected,
this sequence reduces to the short exact sequence
$$
0\to\pi_1(X)\to\pi_1(X\rtimes \Gg)\to s^{-1}(\epsilon(x_0))\to 0
$$
and isomorphisms
$$
\pi_n(X)\to\pi_n(X\rtimes \Gg)
$$
for $n\geq 2$.

(5)
Let $X$ be a semilocally simply connected connected pointed topological space
and let $\Pi_1(X)$ be the fundamental groupoid over $X$.
One can check that $\Pi_1(X)$ is Serre groupoid. We have
$s^{-1}(x_0)=\widetilde{X}$, the universal covering space over $X$.
We get the long exact sequence
$$ 
\ldots\to\pi_n(\widetilde{X})\to\pi_n(X)\to\pi_n(\Pi_1(X))\to\pi_{n-1}(\widetilde{X})\to\ldots
$$
Because $\pi_1(\widetilde{X})=0$ and $\pi_n(\widetilde{X})\to\pi_n(X)$ are isomorphisms
for $n\geq 2$, we have 
$$
\pi_1(\Pi_1(X))=\pi_1(X)
$$
and the other homotopy groups are zero.
We can get the same result in a different
way as well, because $\Pi_1(X)$ is transitive groupoid, thus Morita equivalent to
the groupoid $(\pi_1(X)\rightrightarrows \ast)$ representing the discrete group
$\pi_1(X)$.

(6)
Let $\fol$ be the standard foliation of the open M\"{o}bius band.
The associated  holonomy
groupoid is weakly equivalent to the translation groupoid $\ZZ_2\ltimes (-1,1)$
\cite[p.137]{IntroFolandLie}, therefore 
$$
\pi_1(\Hol(\textnormal{M\"ob},\fol))\cong\ZZ_2
$$
and the other homotopy groups are zero.

(7)
Let $\fol$ be the Kronecker foliation of the torus $T^2$. The associated
holonomy groupoid is weakly equivalent to translation groupoid
$\ZZ\ltimes S^1$ \cite[p.137]{IntroFolandLie}, and we get a short exact sequence 
$$
0\to\ZZ\to\pi_1(\Hol(T^2,\fol)))\to\ZZ\to 0.
$$
The other homotopy groups are zero.

(8)
Let $(M,\fol)$ be a foliated manifold such that every leaf is compact with
finite holonomy. Then by \cite[p.141]{IntroFolandLie} the source map of
the associated holonomy groupoid
$\Hol(M,\fol)$ is a proper map and in fact a fiber bundle
(\cite[p.200]{TopDiffGeo}), thus a Serre fibration.
So for any $a_0\in L\subseteq M$, where $L$ is a leaf of the foliation $\fol$,
we get the long exact sequence
$$
\ldots\to\pi_n(\widetilde{L})\to\pi_n(M)\to\pi_n(\Hol(M,\fol))\to
\pi_{n-1}(\widetilde{L})\to\ldots,
$$
where $\widetilde{L}$ is the holonomy covering space of $L$.

(9)
Let $\Gg$ be the proper \'{e}tale groupoid associated to
an orbifold $Q$ of dimension $n$. Then from
\cite[p.44]{IntroFolandLie} and \cite[p.143]{IntroFolandLie} it follows that
$\Gg$ is Morita equivalent to the translation groupoid
$U(n)\ltimes UF(Q)\rightrightarrows UF(Q)$,
where $UF(Q)$ is the unitary frame bundle associated to the orbifold
$Q$ and $U(n)$ the unitary group.
This translation groupoid is a Serre groupoid,
so we get the associated long exact sequence of homotopy groups 
$$
\ldots\to\pi_n(U(n))\to\pi_n(UF(Q))\to\pi_n(\Gg)\to\pi_{n-1}(U(n))\to\ldots.
$$
\end{Example}

\section{Riemannian foliations}
  
\textsl{
In this section we apply the theory that we have developed so far to
the holonomy groupoid of transversely complete and, more generally,
Riemannian foliations.
For definition and properties of such foliations,
see e.g. \cite{IntroFolandLie,MM2007,Molino1988}. 
}
\vspace{0.25cm}
  
First let us recall the definition of a transverse principal bundle
\cite[p.98]{IntroFolandLie}. 
Let $G$ be a Lie group, $(M,\fol)$ a foliated manifold and
$\pi:E\to M$ a (smooth) principal $G$-bundle with a foliation
$\widetilde{\fol}$ on the total space $E$ such that
\begin{enumerate}
\item[(i)]  $\widetilde{\fol}$ is preserved by the action of $G$, and
\item[(ii)] the projection $\pi:E\to M$ maps each leaf $\widetilde{L}$
of $\widetilde{\fol}$ onto a leaf
$L=\pi(\widetilde{L})$ of $\fol$ such that
the restriction $\pi|_{\widetilde{L}}:\widetilde{L}\to L$ is
the holonomy covering of $L$.
\end{enumerate}
Then we say that $(E,\widetilde{\fol})$ is a transverse principal $G$-bundle
over $(M,\fol)$.

For such a transverse principal $G$-bundle $(E,\widetilde{\fol})$
there is well defined natural projection functor
$$ \Hol(E,\widetilde{\fol})\to\Hol(M,\fol) .$$
If $\gamma$ is a path inside a leaf $L$ of $\fol$,
then, given some initial lift of the starting point, there is a
canonical lift $\widetilde{\gamma}$ of $\gamma$ along $\pi$
which lies inside a leaf of $\widetilde{\fol}$.
This lifting property is well defined on the holonomy classes of paths. 
Indeed,
suppose that $\gamma:(S^1,1)\to (M,m_0)$ is a loop in $L$ through with
trivial holonomy, choose $n_0\in\pi^{-1}(m_0)$ and write $\widetilde{L}$
for the leaf of $\widetilde{\fol}$ with $n_0\in\widetilde{L}$.
The canonical lift $\widetilde{\gamma}$ of $\gamma$ to $\widetilde{L}$ is
then again a loop.
We have to see that $\widetilde{\gamma}$ has trivial holonomy.

Since $\gamma$ has trivial holonomy, there exists a small 
transversal section $T$ of $(M,\fol)$ and a map
$$
\Gamma:S^1\times T\to (M,\fol)
$$
such that $\Gamma|_{S^1\times\{m_0\}}=\gamma$,
$\Gamma(S^1\times \{m\})$ lies in a leaf of $\fol$ for any $m\in T$
and
$\Gamma(\{z\}\times T)$ is a transversal section of $(M,\fol)$
for any $z\in S^1$. We may choose $T$ so small that there exists
a submanifold $\widetilde{T}$ of $E$ such that $e_0\in\widetilde{T}$,
$\pi:\widetilde{T}\to T$ is a diffeomorphism and the tangent space of $E$
at any point $e\in\widetilde{T}$ is a direct sum of the tangent space
of $\widetilde{T}$ at $e$, the tangent space of the fiber of $\pi$
through $e$ and the tangent space of the leaf of $\widetilde{\fol}$
through $e$. By applying the unique lifting property to
the loops $\Gamma(S^1\times\{m\})$, we obtain a lift $\widetilde{\Gamma}$
of $\Gamma$ along $\pi$. Now define a map
$\Psi:S^1\times T\times G\to E$ by
$$ \Psi(z,m,g)=\widetilde{\Gamma}(z,m)g.$$
This implies that $\widetilde{\gamma}=\Psi|_{S^1\times \{m_0\}\times \{1\}}$
has trivial holonomy, because
$\Psi(S^1\times \{m\}\times \{g\}$ lies in a leaf of $\widetilde{\fol}$
for any $m\in T$ and any $g\in G$, while
$\Psi(\{z\}\times T\times G$ is a transversal section of $(E,\widetilde{\fol})$.

\begin{Proposition}\label{Prop:Sf}
Let $\pi:(E,\widetilde{\fol})\to(M,\fol)$ be a transverse principal bundle
over a foliated manifold $(M,\fol)$. 
The projection functor $\Hol(E,\widetilde{\fol})\to\Hol(M,\fol)$
is a Serre fibration both on objects and on arrows, and
the principal bundle associated to this functor is also 
a Serre fibration.
\end{Proposition}
  
\begin{proof}
The projection functor
$\phi:\Hol(E,\widetilde{\fol})\to\Hol(M,\fol)$ equals the Serre fibration
$\pi:E\to M$ on objects. 
The last part of the statement
is a consequence of Proposition \ref{prop:Serrefunctor},
as the map
$(\phi,s):\Hol(E,\widetilde{\fol})_1 \to\Hol(M,\fol)_1 \times_{M}E$
is a Serre fibration because of the unique lifting property of
holonomy classes of paths discussed above.
Since the projection $\Hol(M,\fol)_1 \times_{M}E\to \Hol(M,\fol)_1 $
is also a Serre fibration, it follows that
$\phi:\Hol(E,\widetilde{\fol})_1 \to \Hol(M,\fol)_1 $ is a Serre fibration as well.
\end{proof}

\begin{Theorem} \label{theo:HolSerre}
The holonomy groupoid  $\Hol(M,\fol)$ of any
transversely complete foliation $(M,\fol)$
is a Serre groupoid.
\end{Theorem}

\begin{proof}
In fact, we will prove that the source map $s:\Hol(M,\fol)_1\to M$ is locally
trivial fibration. We know \cite[p.90]{IntroFolandLie} that all the leaves of
foliation $(M,\fol)$ are diffeomorphic and have trivial holonomy. 
Let $L$ be a leaf of $\fol$ and $x_0\in L$.
We can choose complete projectable vector fields
$Y_1,\ldots,Y_q$  on $(M,\fol)$ such that the span of $(Y_1)_{x_0}\ldots,(Y_q)_{x_0}$
is complementary to the tangent space of $\fol$ at $x_0$.
The flows of these vector fields combine to a map
$T:L\times\RR^q\to M$, 
$$
T(x,(t_1,\ldots,t_q))=(e^{t_1Y_1}\circ\cdots\circ e^{t_qY_q})(x),
$$
which has the property that any $T_t=T(\textnormal{-},t)$ is a diffeomorphism
between $L$ and another leaf of $\fol$.
We can choose a small open ball $B\subset\RR^q$ centered at $0$
and an open contractible neighbourhood  $U$ of $x_0$ in $L$
such that $T|_{U\times B}$ is a smooth open embedding and in fact
a foliation chart. Write $W=T(U\times B)$. Observe that
$T(L\times B)$ is an open neighbourhood of $L$ in $M$ which
equals the saturation of $W$.
For any $y\in W$ denote by $(x_{y},t_{y})$ for the element
in $U\times B$ with $T(x_{y},t_{y})=y$.

We need to show that there exists a diffeomorphism
$\phi:W\times s^{-1}(x_0)\to s^{-1}(W)$ over $W$. To construct $\phi$,
let $y\in W$ and let $\gamma$ be a path in $L$ starting at $x_{0}$.
Then $T_{y_{t}}\circ\gamma$ is a path inside the leaf through $y$,
and we define $\phi(y,\gamma)$ to be the holonomy class
of the concatenation of $T_{y_{t}}\circ\gamma$ with a path inside
$T(U\times\{y_t\})$ starting at $y$. The result, of course, depends only
on the holonomy class of $\gamma$ and gives us a well defined smooth
map $\phi$.

To see that $\phi$ is a diffeomorphism, note that $\phi^{-1}$ can be described
in a similar way: For any path $\tau$ representing an element in $s^{-1}(W)$
starting at a point $y\in W$, we have the path
$(T_{y_{t}})^{-1}\circ\tau$ in $L$ starting at a point in $U$.
Then $\phi^{-1}(\tau)$ is given by the pair $(y,\gamma')$,
where $\gamma'$ is the holonomy class of the concatenation of
$\phi^{-1}(\tau)$ with a path in $U$ starting at $x_0$.
\end{proof}

\begin{Theorem}
For any transversely complete foliation
$\fol$ on a connected manifold $M$ with leaf $L$ 
we have a long exact sequence
$$
\ldots\to\pi_n(L)\to\pi_n(M)\to\pi_n(\Hol(M,\fol))\to\pi_{n-1}(L)\to\ldots
$$
\end{Theorem}

\begin{proof}
This follows from Theorem \ref{theo:HolSerre}, Theorem \ref{Thm:LES_serreGPD}
and the fact that the source fibers of $\Hol(M,\fol)$ are diffeomorphic to $L$.
\end{proof}

Let $\fol$ be a Riemannian foliation of codimension $q$
on compact connected manifold $M$. Then
\cite{IntroFolandLie,Molino1988} the canonical
lifted foliation $\widetilde{\fol}$
to the associated transverse orthogonal frame bundle $OF(M,\fol)$ is
transversely parallelizable. Also we have a well defined functor
$\Phi:\Hol(OF(M,\fol),\widetilde{\fol})\to\Hol(M,\fol)$
because $(OF(M,\fol),\widetilde{\fol})\to(M,\fol)$
is a transverse principal $O(q)$-bundle.

\begin{Theorem}
Let $\fol$ be a Riemannian foliation
on compact connected manifold $M$ with base point $m_0$.
Then the the holonomy groupoid $\Hol(M,\fol)$ is
a Serre groupoid in we have a natural long exact sequence
$$
\ldots\to\pi_n(\widetilde{L})\to\pi_n(M)\to\pi_n(\Hol(M,\fol))\to\pi_{n-1}(\widetilde{L})\to\ldots,
$$
where $\widetilde{L}$ is the holonomy cover of the leaf $L$ of $\fol$ through $m_0$.
\end{Theorem} 

\begin{proof}
In the commutative diagram
$$
\xymatrix{
\Hol(OF(M,\fol),\widetilde{\fol})_1\ar[r]^-\Phi\ar[d]_-s & \Hol(M,\fol)_1\ar[d]^-s\\
OF(M)\ar[r]^-\Phi & M,
}
$$
the left vertical map and both horizontal maps are Serre fibrations
(Theorem \ref{theo:HolSerre} and Proposition \ref{Prop:Sf}).
From this, follows that the right $s$ is also a Serre fibration and
we get the associated long exact sequence.
\end{proof}

We can now apply Proposition~\ref{Prop:Sf}
to the case of Riemannian foliation
$(M,\fol)$. We get the following theorem.

\begin{Theorem}
Let $(M,\fol)$ be a Riemannian foliation of codimension $q$ on a compact connected
manifold $M$. Then we have a natural long exact sequence of homotopy groups
$$
\ldots\to\pi_n(O(q))\to\pi_n(\Hol(OF(M,\fol),\widetilde{\fol}))\to
\pi_n(\Hol(M,\fol))\to\pi_{n-1}(O(q))\to\ldots,
$$
where $(OF(M,\fol),\widetilde{\fol})$ is the associated lifted foliation on the the
transverse orthogonal frame bundle and $O(q)$ is the unitary group of order $q$
viewed as a topological space.
\end{Theorem}

\begin{proof}
The long exact sequence is obtained by Theorem \ref{Theorem:LES} and
Proposition~\ref{Prop:Sf}. We only need to identify the fiber
of the principal bundle associated to the functor
$\Hol(OF(M,\fol),\widetilde{\fol})\to\Hol(M,\fol)$, which is
$$ \Gk=\Hol(OF(M,\fol),\widetilde{\fol})\ltimes(OF(M,\fol)\times_{M}s^{-1}(m_{0})),$$
to be Morita equivalent to the topological space $O(q)$.
Here $s^{-1}(m_{0})$ denotes the source fiber of the groupoid
$\Hol(M,\fol)$ over a chosen base point $m_0$ of $M$.
If we choose a base point $e_{0}$ in the fiber
of the projection $\pi:OF(M,\fol)\to M$ over $m_0$,
we obtain a canonical embedding
$\iota:O(q)\to (OF(M,\fol)\times_{M}s^{-1}(m_{0}))$,
$U\mapsto (e_0U,1_{m_0})$.
Now one can easily check that the induced groupoid
$\iota^{\ast}(\Gk)$ is Morita equivalent to $\Gk$ \cite{PoissonGeo}
and equal to the unit groupoid $(O(q)\rightrightarrows O(q))$.
\end{proof}

\section{Effect of an \'{e}tale groupoid}

\textsl{
It is natural to represent an orbifold with a proper \'{e}tale
groupoid \cite{Moerdijk1}. This groupoid may not be effective,
and even if it is, it is often desirable to represent it as
the effect of another, simpler \'{e}tale groupoid. In
this section, we study the connection between the homotopy groups of
a proper \'{e}tale groupoid and the homotopy groups of its associated
effect groupoid.
}
\vspace{0.25cm}

Let $\Gg$ be an \'{e}tale Lie groupoid. We have the functor
$\Eff:\Gg\to\Eff(\Gg)$
between \'{e}tale groupoids $\Gg$ and $\Eff(\Gg)$ \cite{IntroFolandLie}.
This functor is the identity on objects and
surjective local diffeomorphism on arrows.
If the groupoid $\Gg$ is proper, then $\Eff(\Gg)$ is also proper
and the map
$\Eff:\Gg_1\to\Eff(\Gg)_1$ is proper.
Because both $\Gg_1$ and $\Eff(\Gg)_1$ are  Hausdorff manifolds,
this implies that the map $\Eff:\Gg_1\to\Eff(\Gg)_1$
is in fact a covering projection (see e.g.
\cite{HenriquesNONEFF}).

\begin{Proposition}
Let $\Gg$ be a proper \'{e}tale Lie groupoid. Then 
the principal bundle associated to
the functor
$\Eff:\Gg\to\Eff(\Gg)$ is a Serre fibration.
\end{Proposition}

\begin{proof}
We have already seen that $\Eff$ is a covering projection on arrows,
while it is the identity on objects. The proposition now follows from 
Proposition \ref{prop:Serrefunctor}.
\end{proof}

For a proper \'{e}tale Lie groupoid $\Gg$,
let us denote by $\Gg_{a_0}^0$ the group of ineffective arrows in the
isotropy group of $\Gg_1$ at a point $a_0\in\Gg_0$,
namely all such arrows $g\in\Gg_1$ from $a_0$ to $a_0$
such that $\Eff(g)=1_{a_0}$.

\begin{Theorem}\label{Thm:Eff}
Let $(\Gg,a_0)$ be a pointed proper \'{e}tale Lie groupoid.
Then the $\Eff$ functor
induces isomorphisms
$\pi_n(\Gg,a_0)\cong\pi_n(\Eff(\Gg),a_0)$, for
$n=0,3,4,5,\ldots$, and there is an exact sequence 
$$
0\to\pi_2(\Gg)\to\pi_2(\Eff(\Gg))\to
\Gg_{a_0}^0\to\pi_1(\Gg)\to\pi_1(\Eff(\Gg))\to 0.
$$
\end{Theorem}
  
\begin{proof}
We have already seen that the principal bundle associated to the
functor $\Eff$ is a Serre fibration, so there is
the associated long exact sequence as in Theorem \ref{Theorem:LES}.
The fiber appearing in this exact sequence is
the groupoid $\Gg\ltimes\sigma^{-1}(a_0)$, where $\sigma$ is the source map of the
groupoid $\Eff(\Gg)$. One can check that
this groupoid is transitive and therefore
weakly equivalent to a group, which in this case
exactly the discrete group $\Gg_{a_0}^0$.
It follows that $\pi_1(\Gg_{a_0}^0\rightrightarrows \ast)=\Gg_{a_0}^0$,
while $\pi_n(\Gg_{a_0}^0\rightrightarrows \ast)=0$ for $n\neq 1$.
\end{proof}
  
\begin{Example} \rm
Let $(M,m_0)$ be a pointed manifold and
let $\Gamma$ be a discrete group acting smoothly on $M$ in such a way
that the translation groupoid $\Gamma\ltimes M$ is proper. 
This groupoid is a Serre
groupoid, so its effect $\Eff(\Gamma\ltimes M)$ is also a Serre groupoid.
Using Theorem \ref{Thm:LES_serreGPD} for both groupoids,
we see that $\pi_2(\Gamma\ltimes M)\cong\pi_2(M)$
and $\pi_2(\Eff(\Gamma\ltimes M))\cong\pi_2(M)$; furthermore, the
induced isomorphism
$\pi_2(\Gamma\ltimes M)\cong\pi_2(\Eff(\Gamma\ltimes M)$
is exactly $\pi_2(\Eff)$. From Theorem \ref{Thm:Eff}
we get isomorphisms
$$\pi_n(\Gamma\ltimes M)\cong\pi_n(\Eff(\Gamma\ltimes M)),$$
for $n\neq 1$, and a short exact sequence
$$
0\to \Gamma_{m_0}^0\to\pi_1(\Gamma\ltimes M)\to
\pi_1(\Eff(\Gamma\ltimes M))\to 0,
$$
where $\Gamma_{m_0}^0$ is the subgroup of $\Gamma$
containing all the elements $g\in \Gamma$ which act on $M$
by a diffeomorphism with trivial germ at $m_0$.
\end{Example}


\begin{thebibliography}{10}

\bibitem{Bott}
Raoul Bott.
\newblock Lectures on algebraic and differential topology.
\newblock {\em Lecture notes in mathematics}, 279(2):1--95, 1972.

\bibitem{Chen}
Weimin Chen.
\newblock On a notion of maps between orbifolds ii: homotopy and cw-complex.
\newblock {\em Communications in Contemporary Mathematics}, 8(6):763--821,
  2006.

\bibitem{Dold}
Albrecht Dold.
\newblock Partitions of unity in the theory of fibrations.
\newblock {\em Annals of Mathematics}, 78(2):223--255, 1963.

\bibitem{Haefliger}
Andre Haefliger.
\newblock Manifolds - Amsterdam 1970.
\newblock {\em Lecture notes in mathematics}, 197(2):133--164, 1970.

\bibitem{HaefligerClass}
Andre Haefliger.
\newblock Structure transverse des feuilletages.
\newblock {\em Asterisque}, 116:70--97, 1984.

\bibitem{AT}
Allen Hatcher.
\newblock {\em Algebraic Topology}.
\newblock Cambridge University Press, 2001.

\bibitem{HenriquesNONEFF}
Andre Henriques and David~S. Metzler.
\newblock Presentations of noneffective orbifolds.
\newblock {\em Trans. Amer. Math. Soc.}, 356:2481--2499, 2004.

\bibitem{FibBund}
Dale Husemoller.
\newblock {\em Fibre Bundles}.
\newblock Springer--Verlag, 1994.

\bibitem{TopDiffGeo}
Peter~W. Michor.
\newblock {\em Topics in Differential Geometry}.
\newblock AMS, 2008.

\bibitem{ClassTopoi}
I.~Moerdijk.
\newblock {\em Classifying Spaces and Classifying Topoi}.
\newblock Springer--Verlag, 1995.

\bibitem{Moerdijk1}
I.~Moerdijk.
\newblock Orbifolds as groupoids: an introduction.
\newblock In {\em Orbifolds in mathematics and physics ({M}adison, {WI},
  2001)}, volume 310 of {\em Contemp. Math.}, pages 205--222. Amer. Math. Soc.,
  Providence, RI, 2002.

\bibitem{IntroFolandLie}
I.~Moerdijk and J.~Mr\v{c}un.
\newblock {\em Introduction to Foliations and Lie Groupoids}.
\newblock Cambridge University Press, 2003.

\bibitem{PoissonGeo}
I.~Moerdijk and J.~Mr{\v{c}}un.
\newblock Lie groupoids, sheaves and cohomology.
\newblock In {\em Poisson geometry, deformation quantisation and group
  representations}, volume 323 of {\em London Math. Soc. Lecture Note Ser.},
  pages 145--272. Cambridge Univ. Press, Cambridge, 2005.

\bibitem{MM2007}
I.~Moerdijk and J.~Mr\v{c}un.
\newblock {On the developability of Lie subalgebroids.}.
\newblock {\em Adv. Math.}, 210:1–-21, 2007. 
  
\bibitem{Molino1988}
P.~Molino.
\newblock {\em Riemannian foliations}, volume~73 of {\em Progress in
  Mathematics}.
\newblock Birkh\"auser Boston Inc., Boston, MA, 1988.
\newblock Translated from the French by Grant Cairns, With appendices by
  Cairns, Y. Carri\`ere, \'E. Ghys, E. Salem and V. Sergiescu.

\bibitem{HilSkan}
Janez Mr\v{c}un.
\newblock {\em Stability and Invariants of Hilsum-Skandalis Maps}.
\newblock PhD, 1996.

\bibitem{DiffGeoFol}
Bruce~L. Reinhart.
\newblock {\em Differential geometry of foliations}.
\newblock Springer--Verlag, 1983.

\bibitem{Thurston}
William~P. Thurston.
\newblock {\em Three-Dimensional Geometry and Topology}.
\newblock University of Minnesota, 1990.

\bibitem{Tsuboi}
Takashi Tsuboi.
\newblock Classifying spaces for groupoid structures.
\newblock In {\em Foliations, geometry, and topology},
  volume 498 of {\em Contemp. Math.}, pages 67–-81. Amer. Math. Soc., 
  Providence, RI, 2009.

\end{thebibliography}
\end{document}